\documentclass[12pt]{amsart}
\usepackage[all]{xy}
\usepackage{amssymb}
\usepackage{algorithmicx}
\usepackage{tikz}
\usepackage{graphics}
\usepackage{graphicx}
\graphicspath{ {images/} }
\usepackage{epstopdf}
\calclayout                   
\newtheorem{dummy}{anything}[section]
\newtheorem{theorem}[dummy]{Theorem}

\newtheorem{lemma}[dummy]{Lemma}

\theoremstyle{definition}

\newtheorem{remark}[dummy]{Remark}

\newcommand{\bZ}{\mathbb Z}

\newcommand{\bR}{\mathbb R}

\def\:{\mkern 1.2mu \colon}
\newcommand{\mmatrix}[4]{\left (\vcenter
{\xymatrix@C-2pc@R-2pc{#1&#2\\#3&#4} } \right )}

\numberwithin{equation}{section} 

\begin{document}
\title[Decomposing Perfect Discrete Morse Functions]
{Decomposing Perfect Discrete Morse Functions on Connected Sum of $3$-Manifolds}

\subjclass[2010]{Primary:57R70, 37E35; Secondary:57R05}
\keywords{Perfect discrete Morse function, discrete vector field, connected sum.}

\author{Ne\v{z}a Mramor Kosta, Mehmetc{\.I}k Pamuk and Han{\.I}fe Varl{\i}}

\address{Faculty of Computer and Information Science 
\newline\indent
and Institute of Mathematics, Physics and Mechanics, 
\newline\indent
University of Ljubljana
\newline\indent
Ljubljana, Jadranska 19, Slovenia} \email{neza.mramor@fri.uni-lj.si}

\address{Department of Mathematics
\newline\indent
Middle East Technical University
\newline\indent
Ankara 06531, Turkey} \email{mpamuk{@}metu.edu.tr}

\address{Department of Mathematics
\newline\indent
Middle East Technical University
\newline\indent
Ankara 06531, Turkey}  \email{varlihanife55@gmail.com}

\date{\today}

\begin{abstract}\noindent 
In this paper, we show that if a closed, connected, oriented $3$-manifold $M=M_{1}\# M_{2}$ admits a perfect discrete Morse 
function, then one can decompose this function as perfect discrete Morse functions on $M_1$ and $M_2$.  We also give an explicit construction  
of a separating sphere on $M$ corresponding to such a decomposition.		
\end{abstract}

\maketitle

\section{INTRODUCTION}
In the $1990$s, Robin Forman \cite{forman1} proposed a discrete version of Morse theory called discrete Morse 
theory.  Discrete Morse theory provides a tool for studying the topology of discrete objects via critical cells of a discrete Morse function, which is in many ways analogous to the study of the topology of a manifold via critical points of a differentiable function on it.  
A discrete Morse function on a cell complex is an assignment of a real number to each cell in such a way that the natural 
order given by the dimension of cells is respected, except at most in one (co)face for each cell.

A perfect discrete Morse function is a very special discrete Morse function where the number of critical cells in each dimension is equal to the corresponding Betti number of the complex. Such functions are very useful since they provide an efficient way of computing homology (\cite{GV}) and have many other applications (see \cite{KB} -- \cite{JK}).
The purpose of this paper is to study $\bZ$-perfect discrete Morse functions on the connected sum of triangulated closed, connected and oriented 
$3$-manifolds.  It extends the results of \cite{NMH}, where a proof in the $2$-dimensional case was given. 
In dimension $2$, we showed that a $\bZ$-perfect discrete Morse function on a connected sum of closed, connected and oriented surfaces 
can be decomposed as a $\bZ$-perfect discrete Morse function on each summand \cite[Theorem 4.4]{NMH}.  We also gave a method for constructing a 
separating circle on the connected sum 
\cite[Theorem 4.6]{NMH}.  
In this paper, we consider the problem of decomposing $\bZ$-perfect discrete Morse functions on the connected sum $M=M_1\# M_2$ of closed, connected, oriented, 
triangulated $3$-manifolds.  We give an explicit construction of a separating sphere $S$, such that one can easily obtain each prime factor of the connected sum.  
Our methods do not extend to higher dimensions directly for several reasons. First of all, we rely strongly on duality which connects 1 and 2-dimensional homology generators in a way, specific for dimension 3. Second, classifyng the separating manifold as sphere is much more difficult. Also, the decomposition of a manifold 
into its prime factors is not unique in higher dimensions.  

In dimension $3$, it is known that not all $3$-manifolds admit a perfect discrete Morse function. 
For example, a $3$-manifold whose fundamental group contains a torsion subgroup, such as any lens space, cannot admit any $\bZ$-perfect discrete Morse function \cite{Rayala}.
On the other hand, there is a large class of $3$-manifolds that do admit such a function, for example any 3-manifold of the type $F\times S^1$ where $F$ is a closed, connected, oriented surface, or any connected sum of such components. Our methods extend to perfect discrete Morse functions over other coefficients which are supported by other classes of manifolds.

The paper is organized as follows: In section~$2$, we recall some basic notions of discrete Morse theory and introduce notation.  In Section~$3$ we construct a separating sphere on the connected sum of closed, connected, oriented and triangulated $3$-manifolds with a certain condition on the discrete vector field. Finally, in Section~$4$ we show how to  decompose a $\bZ$-perfect discrete Morse function on the connected sum as $\bZ$-perfect discrete Morse functions on each summand. 



\section{Preliminaries}
In this section we recall necessary basic notions of discrete Morse theory.  For more details we refer the reader to \cite{forman1} and \cite{forman2}.
Throughout this section $K$ denotes a finite regular cell complex.  We write $\tau > \sigma$ if $\sigma \subset \overline{\tau}$ (closure of $\tau$). 
 
A real valued function $f\colon K\rightarrow \bR$ on $K$ is a discrete Morse function if for any $p$-cell $\sigma \in K$, there is at most one 
$(p+1)$-cell $\tau$ containing $\sigma$ in its boundary such that $f(\tau)\leq f(\sigma)$ and there is at most one 
$(p-1)$-cell $\nu$ contained in the boundary of $\sigma$ such that $f(\nu)\geq f(\sigma)$.  A $p$-cell $\sigma \in K$ is a critical $p$-cell of $f$ if
$f(\nu)< f(\sigma)< f(\tau)$ for any $(p-1)-$face $\nu$ and $(p+1)-$coface $\tau$ of $\sigma$.  A cell is regular if it is not critical.
Regular cells appear in disjoint pairs. They are usually denoted by arrows pointing in the direction of increasing dimension and decreasing function values and form the gradient vector field

$$
V = \{(\sigma\to \tau) \mid \dim \sigma=\dim \tau-1, \sigma<\tau, f(\sigma)\geq f(\tau)\}.
$$

In this notation, a cell is critical if and only if it is neither the head nor the tail of an arrow. 
\noindent

 

A gradient path or a $V$-path of dimension $(p+1)$ is a sequence of cells 
$$
\sigma_{0}^{(p)}\to \tau_{0}^{(p+1)}>\sigma_{1}^{(p)}\to \tau_{1}^{(p+1)}>\cdots \to \tau_{r}^{(p+1)}>\sigma_{r+1}^{(p)},
$$  
such that $(\sigma_{i}^{(p)}\to \tau_{i}^{(p+1)})\in V$ and $\sigma_{i}^{(p)} \neq \sigma_{i+1}^{(p)} < \tau_{i}^{(p+1)}$ for each $i = 0, 1, \ldots,  r$.  

A gradient path is represented by a sequence of arrows. Since function values along a gradient path of length more than two decrease, 
gradient paths cannot form cycles, and every such collection of arrows on a cell complex that do not form cycles corresponds to a discrete Morse function.  
In applications of discrete Morse theory it often suffices to consider the gradient vector field underlying a given discrete Morse function 
instead of the specific function values.

In this paper, we work with discrete Morse functions, where the number of critical $i$-cells 
equals the $i$-th Betti number of the complex.  Such discrete Morse functions are called perfect (with respect to the given coefficient ring, 
which is fixed as $\bZ$ throughout the paper and suppressed in our notation).

Note that if $M$ is a closed, connected oriented $n$-manifold and $f$ is a perfect discrete Morse function on it, then it has a single critical $n$-cell and a single critical vertex. In addition, since 1-paths cannot split, every vertex is connected to the critical vertex by a unique path, and since paths of maximal dimension in a manifolds cannot merge, every cell of codimension $1$ is connected to the critical $n$-cell by a unique path starting in the boundary of the critical cell and ending in this cell. 



\section{Separating the components of the connected sum}

In this section we first discuss  how to group the critical cells i.e., which critical cells must belong to the same component.  We consider $M$ as the union of two submanifolds which we denote by $M-M_1$ and $M-M_2$ glued along a common boundary $S$ which bounds discs $D_1$ in $M_1$ and $D_2$ in $M_2$. In our construction, the critical $3$-cell will belong to $M-M_{1}$ and the critical $0$-cell to the $M-M_2$. Next we construct a separating sphere for the connected sum  with a certain arrow configuration on it (cf. \cite{NMH}).  
After grouping the critical cells the vector field will produce paths only from one component to the other.  Finally, we explain how to extend the vector fields to $M_1$ and 
$M_2$.

The first step in constructing the separating sphere $S$ is to decide which critical cells should belong to the same component. Since $M$ is a connected orientable closed $3$-manifold $H_3(M; \bZ)\cong \bZ$, and since $f$ is perfect there is a unique critical $3$-cell and a unique critical vertex.  Note  that the gradient paths of the discrete Morse function determine a Morse-Smale type decomposition of the manifold \cite{GN}. The union of $2$-paths starting in a given critical $2$-cell forms a geometric representative of a $2$-dimensional homology generator.  On the other hand, each of the two cofaces of a critical 2-cell is at the end of a unique 3-path starting in the critical 3-cell. These $3$-paths together with the critical $3$-cell at the begining and the critical $2$-cell at the end form a solid torus ($S^1\times D^2$) whose core represents the dual $1$-dimensional homology generator. These two classes intersect transversally with intersection number 1, and belong to the same component in the decomposition of $M$ (existence and uniqueness of such pairs follows from 
Poincare duality).  

In order to detect, which critical $1$-cell generates the $1$-dimensional homology class homologous to the core of the above solid torus, we consider all the $2$-paths that end at a given critical $1$-cell.  
Consider
the dual decomposition in which a critical $1$-cell corresponds to a critical $2$-cell.  Instead of counting the number of intersections of homology classes given by the critical cells 
and gradient paths (which might not be possible since paths might merge), we count the number of intersections of the geometric representatives of the dual homology classes, perturbed, if necessary, so that they intersect transversally.  It follows from duality and the structure of the cohomology ring of a connected sum that if the representatives of these dual homology classes intersect an odd number of times then the corresponding critical $1$ and $2$-cells belong to the same component. 

\begin{remark}
In the figures we demonstrate our constructions on cubical complexes since the subdivisions can be shown more clearly in cubical complexes then in simplicial complexes.
\end{remark}

\begin{theorem}\label{decompose3}
Let $M$ be a connected sum of closed, connected, oriented and triangulated  $3$-manifolds $M_1$ and $M_2$ with a $\bZ$-perfect discrete Morse function $f$. 
Then, after possibly some local subdivisions, we can find a separating $2$-sphere $S$ on $M$ such that $M$ is the union of triangulated manifolds
$M-M_1$ and $M-M_2$ with common boundary $S$. In addition the cells on $S$ are never paired with the cells on $M-M_1$.
\end{theorem}

\begin{proof}
Let $C$ be the set of critical $1-$ and $2-$cells of $f$ that belong to the same part of $M$,
and $P_{2}$ be the set of $2$-paths that end at the critical $1$-cells in $C$.  

Let $M-M_1$ contain the unique critical 
$3$-cell, the $3$-paths from the critical $3$-cell to the $2$-cells in $C$ and the 2-paths in $P_2$. Let $S$ be the boundary of $M-M_1$, and let $M-M_2$ denote the closure of the remaining part of $M$.  

In the following lemmas and the remaining part of the proof of the theorem
we modify these parts but by abuse of notation we keep calling them $M-M_1$ and $M-M_2$. Note that a cell and the star of a cell will be a closed cell or star, respectively,  in the remaining part of the proof.



\begin{lemma}\label{3form2}
If different $3$-paths meet along a common $2$-path in the boundary of both then, after some local subdivisions, these two $3$-paths can be separated so that their closures do not contain the $2$-path any more.
\end{lemma}

\begin{proof}
Let $\gamma$ be the common $2$-path, $R$ the union of stars of cells in $\gamma$ and $R'=R \cap (M-M_1)$. 
To separate the  $3$-paths, we follow the steps below in the given order (see also Figure ~\ref{fig:critical3subdivide}):
\begin{enumerate}
\item We first bisect all the $1$-cells in $R'$ that intersect $\gamma$ at a single vertex.
\item We then extend the vector field to the subdivided cells by pairing each vertex bisecting a $1$-cell either with its coface in the star of $\gamma$ in case the coface is not paired with a  $0$-cell on $\gamma$, or with its coface outside the star of $\gamma$ if it is paired with a $0$-cell on $\gamma$.
\item Next we bisect the $2$-cells in $R'$ by connecting the newly introduced vertices in Step (i).  
\item We pair the $1$-cells bisecting the $2$-cells with their cofaces that intersect with $\gamma$.
\item We pair the remaining $2$-cells with their cofaces in the star of $\gamma$ in the subdivided $M-M_1$.
\end{enumerate}

Figure ~\ref{fig:critical3subdivide} shows an example. The blue arrows denote the $3$-paths that end at a critical $2$-cell in $C$ and contain 
the $2$-path $\gamma$ on their boundary.  The front and back faces of these $3$-paths are contained in the boundary of $M-M_1$. The arrow pointing from the face of $\gamma$ into $\gamma$ is an arrow pointing from the  boundary of $M-M_1$ into its interior.
The figure on the right shows the result of the separation of these $3$-paths where 
the red arrows denote the extension of the vector field.  The cells that are colored gray also belong to 
a part of the resulting boundary of  $M-M_1$. Since $M-M_1$ contains only cells in $3$-paths starting in the critical 3-cell and ending in the critical $2$-cells in $C$ and at the
$2$-cells in $P_2$, the 
$2$-cell $\gamma$ is not in its boundary any more, and the arrow pointing inwards has been eliminated. 

\begin{figure}[hbtp]
				\centering
				\includegraphics[width=.90\textwidth,height=.3\textheight]{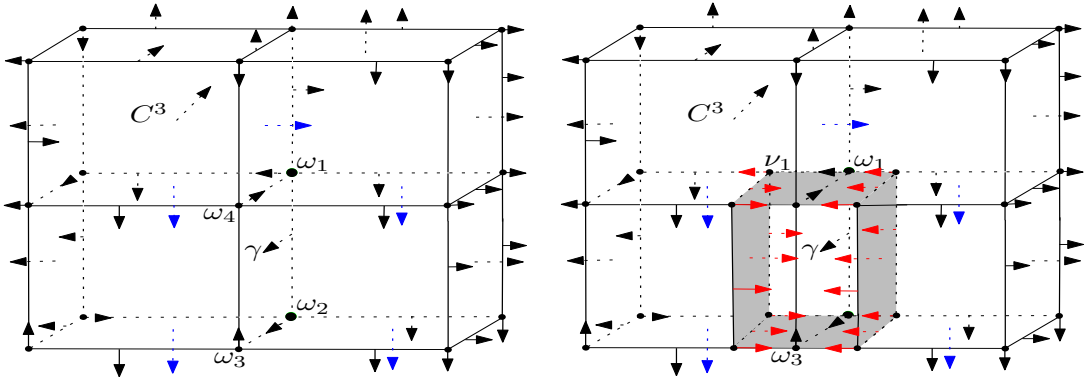}
				\vspace{-3.em}
				\caption{A separation of the $3$-paths that contain the $2$-path $\gamma$ on their common boundaries.}
				\label{fig:critical3subdivide}
\end{figure}

\end{proof}

\begin{lemma}\label{3form3}
Two $3$-paths that meet along a common $1$-path in the boundary of both can, after some local subdivisions, be separated.   
\end{lemma}

\begin{proof}
Let $S$ be the star of the cells on an interior $1$-path $\mu$ in $M-M_1$ that is in the common boundary of different $3$-paths in $M-M_1$ and let 
$S'=S\cap (M-M_1)$. 
Let $\sigma$ and $\sigma'$ be the initial and terminal vertices of $\mu$ in $M-M_1$.  
In order to separate the $3$-paths, we subdivide $S'$ as in the following way (see also Figure ~\ref{fig:critical14}):

\begin{enumerate}
\item We bisect all the $1$-cells in $S'$ that intersect $\mu$ at a single vertex and
pair the vertices bisecting the $1$-cells with their cofaces in the star of $\mu$ if $\sigma'$ is
not paired with a $1$-cell in $S'$.  
\item If $\sigma'$ is paired with a $1$-cell $\lambda$ in $S'$, then we pair the vertex bisecting $\lambda$ with its unpaired coface and 
pair the remaining vertices as in step(i).
\item We bisect the $2$-cells in $S'$ by connecting the newly introduced vertices in step(i). 
\item We extend the vector field to the subdivided cells by pairing the remaining unpaired $1$-cells with their cofaces that intersect $\mu$.
\item Finally, we pair the remaining $2$-cells with their cofaces in the star of $\mu$ obtained after the bisections in $M-M_1$.
\end{enumerate}

Note that all the $3$-paths in the star of $\mu$ end at regular $2$-cells at the end of above steps, and therefore do not belong to $M-M_1$.  
Thus the $3$-paths are separated in the sense that $\mu$ is not contained in their boundary, and is not contained in the boundary of $M-M_1$.  Figure ~\ref{fig:critical14} is an illustration of this separation process, 
where the orange colored arrows represent the $2$-paths in $P_2$ that end at the orange colored critical $1$-cell in $C$ in the top right corner, and blue colored 
arrows denote the $3$-paths that end at the $2$-cells in $P_2$ and admit the $1$-path $\mu$ on their common boundary.  The front and back faces 
on the boundary of these $3$-paths on the left figure belong to the boundary of $M-M_1$, so the arrow from $\sigma$ to $\mu$ points from the boundary of $M-M_1$ into its interior. 
The figure on the right shows the separation of these $3$-paths 
such that the boundary of these paths does not contain $\mu$ any more.  The gray colored cells on the right figure also belong to the 
boundary of $M-M_1$.
 
\begin{figure}[hbtp]
				\centering
				\includegraphics[width=.8\textwidth,height=.25\textheight]{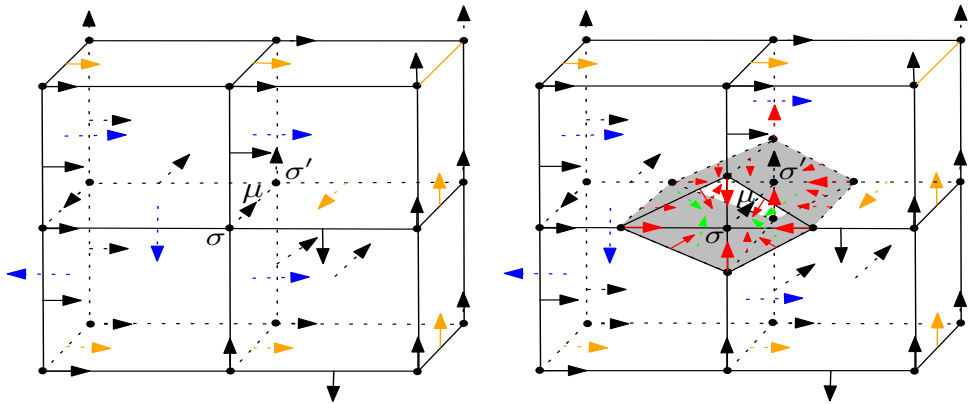}
				\caption{A separation of the $3$-paths that contain a $1$-path on their common boundary.}
				\label{fig:critical14}
\end{figure}

\end{proof}

To complete the proof of the theorem, let $M-M_1$ and $M-M_2$ be the subcomplexes of $M$ obtained after applying Lemma ~\ref{3form2} and Lemma ~\ref{3form3} to   
separate all the $3$-paths which give rise to the existence of common boundary cells in $M-M_1$. It follows that none of the cells on the boundary of $M-M_1$ are paired with interior cells.  

However, the boundary may not be a manifold, that is, there may be some non-manifold 
edges and vertices on the boundary since different $3$-paths may have some edges or vertices on their common boundary.
Moreover, there may be some critical cells which belong to $M-M_2$ on the resulting boundary.

To get rid of a non-manifold cell or a critical cell on the boundary, we bisect the cells in its star. In this way we separate the $3$-paths
that share a non-manifold cell on their boundary or push a critical cell into $M-M_2$.  Then we extend the vector field to the subdivided cells as in 
Lemma~\ref{3form2} and Lemma~\ref{3form3}. In Figure~\ref{fig:nonmanifoldedge} the $3$-paths denoted by blue arrows that intersect in the non-manifold edge denoted by $\tau$ are separated in this way.  The boundary of the paths denoted by blue arrows represents a part of the boundary of $M-M_1$ on the left figure. On the right
the gray colored cells with the boundary of the blue $3$-paths denote a part of the boundary of $M-M_1$, and the red and green arrows show the extension of the vector field to the subdivided cells.

\begin{figure}[hbtp]
\centering
\includegraphics[width=.8\textwidth,height=.15\textheight]{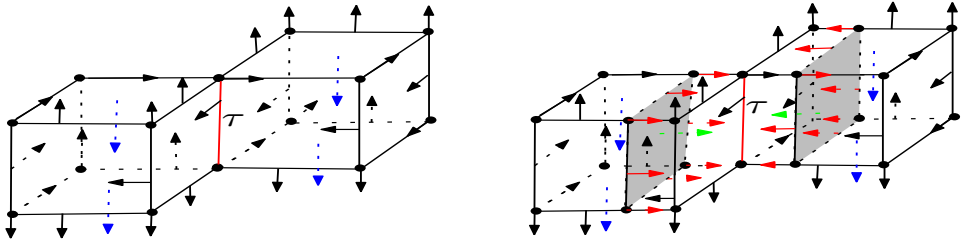}
\caption{A non-manifold edge $\tau$ on the boundary and a separation of the $3$-paths meeting at $\tau$.}
\label{fig:nonmanifoldedge}
\end{figure}

After these modifications the resulting boundary of $M-M_1$ is a $2$-manifold such that none of the cells on it are paired with interior cells.  
Let $V$ be the discrete gradient vector field on $M$ obtained after the necessary subdivisions in the above steps.
Observe that both $M-M_1$ and $M-M_2$ are $3$-manifolds with boundary,  and the restriction of $V$ to $M-M_2$ has no boundary critical cells.  
The numbers of critical cells of $V|_{M-M_2}$ are 
\begin{eqnarray*}
m_0(V|_{M-M_2}) &=& b_0(M_{1})=1,\\
m_1(V|_{M-M_2}) &=& b_1(M_{1}),\\
m_2(V|_{M-M_2}) &=& b_2(M_{1}).
\end{eqnarray*} 


Let $2(M-M_2)$ denote the double of $M-M_2$,  which is a closed $3$-manifold obtained by gluing the two copies of $M-M_2$ along the 
boundaries by the identity map.  Thus we have 
$$	
0=\chi(2(M-M_2)) =2\chi(M-M_2)-\chi(\partial(M-M_2)) 
$$	
and by \cite[Corollary 3.7]{forman1}, we have
$$
\chi(M-M_2) = 1- b_1(M_{1})+b_2(M_{1})=1.
$$
The above equations together imply that $\chi(\partial(M-M_2))=2$.  Since $\partial(M-M_2)$ is a closed, connected, 
oriented $2$-manifold, by classification of surfaces, it is a $2$-sphere $S^{2}$.

Finally, the sphere $S=S^2$ is a separating sphere for the two components $M-M_{1}$ and $M-M_{2}$ in the connected sum. 

\end{proof}

\section{Decomposing the perfect discrete Morse function}
Let $\widetilde{f}$ be a perfect discrete Morse function corresponding to the discrete vector field $V$ on the locally modified $M$ which  
coincides with $f$ except in a neighborhood of the separating sphere $S$.  
Now we can decompose $\widetilde{f}$ into perfect discrete Morse functions on the summands $M_1 \cong (M-M_2)\cup D_1$ and $M_2 \cong (M-M_1)\cup D_2$.


\begin{theorem} \label{extend3}
Given a perfect discrete vector field $V$ on $M=M_1\# M_2$ and a separating sphere $S=\partial (M-M_1)=\partial(M-M_2)$ such that there are no arrows pointing from $S$ into $M-M_1$, we can extend $ V|_{M-M_{2}}$ and $ V|_{M-M_{1}}$ to $M_{1} $ and $M_{2}$ as perfect discrete gradient vector fields which coincide with $V$ everywhere except possibly on a neighborhood of the separating sphere.
\end{theorem}

\begin{proof}
We begin the proof with an observation on $2$-paths between critical $2$- and $1$-cells.  Consider $1$- and $2$-dimensional homology classes that are dual to 
critical $2$- and $1$-cells respectively.  Recall that a critical $2$- and $1$-cell pair belong to the same part if their duals intersect an odd number of times.  If their duals intersect an even number of times, then there cannot be a $2$-path between the critical cells.  To see this one should observe that, 
if there is a path between a critical $2$-cell  and a critical $1$-cell, then the $2$-dimensional classes obtained by following the $2$-paths emanating from a critical $2$-cell 
and considering all the $2$-paths that end at a critical $1$-cell are homologous (the number of intersections of these spheres with 
the core of the solid torus, the dual $1$-dimensional homology class, are the same).

We first  extend $ V|_{M-M_{2}}$ to a discrete gradient vector field on $M_{1}=(M-M_{2})\cup_{S}D_{1}$ where $D_1$ is a triangulated 
$3$-disk with boundary $S$ and with a unique interior vertex.  Let $\Delta$ be a $3$-cell in $D_1$.  Obviously $(M_{1}-\text{int}(\Delta))$ collapses to $(M-M_2)$.  
By the construction of $S$ in Theorem ~\ref{decompose3}, $ V|_{M-M_{2}}$ has no boundary critical cells.  
Therefore, we can extend $ V|_{M-M_{2}}$ to a discrete gradient vector field $V_1$ on $M_{1}-\text{int}(\Delta)$ 
without creating any extra critical cells \cite[Lemma 4.3]{forman1}.  Let $g$ be a discrete Morse function corresponding 
to $V_1$ on $M_{1}-\text{int}(\Delta)$.  Then the function $\widetilde{g}:M_1\rightarrow \bR$ defined as 
	$$\widetilde{g}(\alpha)=\left\{\begin{array}{lll}
		g(\alpha)&; & \alpha \in M_1-\textrm{Int}(\Delta)\\
		1+\textrm{max}\{g(\partial \alpha)\} &; & \alpha=\Delta
		\end{array}\right.$$
is a perfect discrete Morse function on $M_1$ with unique critical $3$-cell $\Delta$.


Next, we extend $ V|_{M-M_{1}}$ to a discrete gradient vector field on $M_{2}=(M-M_{2})\cup_{S}D_{2}$ where $D_2$ is a triangulated disk 
with boundary $S$ and with an interior vertex $\omega$, in the following way:

\begin{enumerate}
\item For each boundary critical cell $\alpha$ on $M-M_1$, we form a pair $(\alpha,\alpha')$ where $\alpha'$ is the coface of $\alpha$ in $\textrm{int}(D_2)$.
\item For each $(\sigma, \beta)\in  V|_{M-M_{1}}$, we form a corresponding pair $(\sigma', \beta')$ where $\sigma'$ and $\beta'$ are 
the cofaces of $\sigma$ and $\beta$ in $\textrm{int}(D_2)$, respectively.
\end{enumerate}

The vertex $\omega$ remains unpaired, and it is the unique critical $0$-cell of the obtained extension $V_2$ of $ V|_{M-M_{1}}$ to $M_2$.  
Clearly, $V_2$ is a discrete gradient vector field induced by a perfect discrete Morse function on $M_2$ with the unique critical $0$-cell $\omega$.

Figure~\ref{fig:sphere0vector} is an example of a discrete gradient vector field $(V|_{M-M_{1}})|_S$ 
where $v_{9}$ is a critical $0$-cell, $e_{1}$, $e_{2}$ and $e_{3}$ are critical $1$-cells, and $\sigma_{1}$,
$\sigma_{2}$, $\sigma_{3}$ and $\sigma_{4}$ are critical $2$-cells. 

\begin{figure}[hbtp]
		\centering
		\includegraphics[width=.6\textwidth,height=.4\textheight]{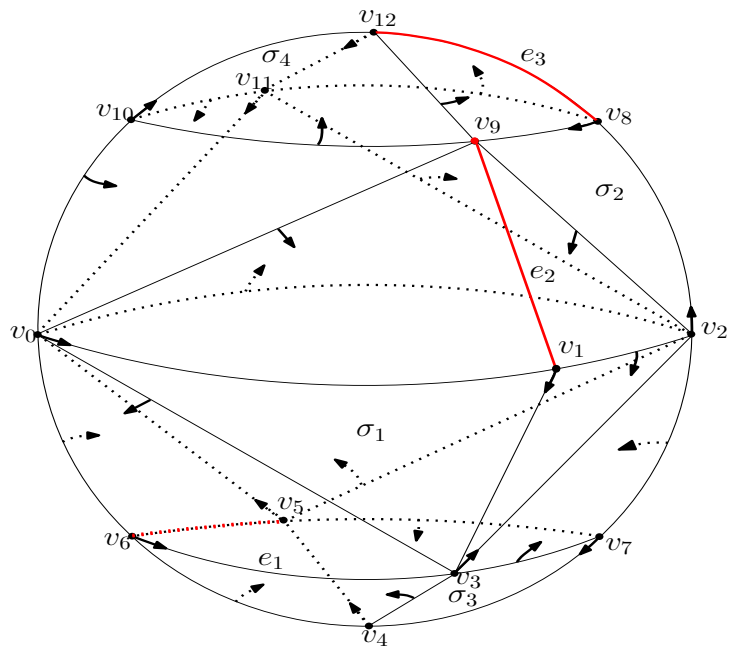}
		\caption{A discrete gradient vector field on a sphere $S^{2}$.}
		\label{fig:sphere0vector}
	\end{figure}
	
The vector field on Figure ~\ref{fig:sphere0extend} denotes the extension 
of $(V|_{M-M_{1}})|_S$ to $D_{2}$
with only one critical $0$-cell which is the vertex $\omega$.  On this extension,  $v_{9}$
is paired with the $1$-cell $[v_{9},\omega]$, $e_{1}$, $e_{2}$ and $e_{3}$ are paired with 
the $2$-cells $[v_{5}, v_{6}, \omega]$, $[v_{1}, v_{9}, \omega]$, $[v_{8}, v_{12}, \omega]$, respectively, 
and $\sigma_{1}$, $\sigma_{2}$, $\sigma_{3}$,
and $\sigma_{4}$ are paired with the $3$-cells $[v_{0}, v_{1}, v_{3}, \omega]$, $[v_{2}, v_{8}, v_{9}, \omega]$, 
$[v_{3}, v_{4}, v_{7}, \omega]$ and $[v_{10}, v_{11}, v_{12}, \omega]$, respectively.  
The pairs of the remaining interior cells depend on the pairs of the cells on 
their faces on $S$.

\begin{figure}[hbtp]
			\centering
			\includegraphics[width=.6\textwidth,height=.4\textheight]{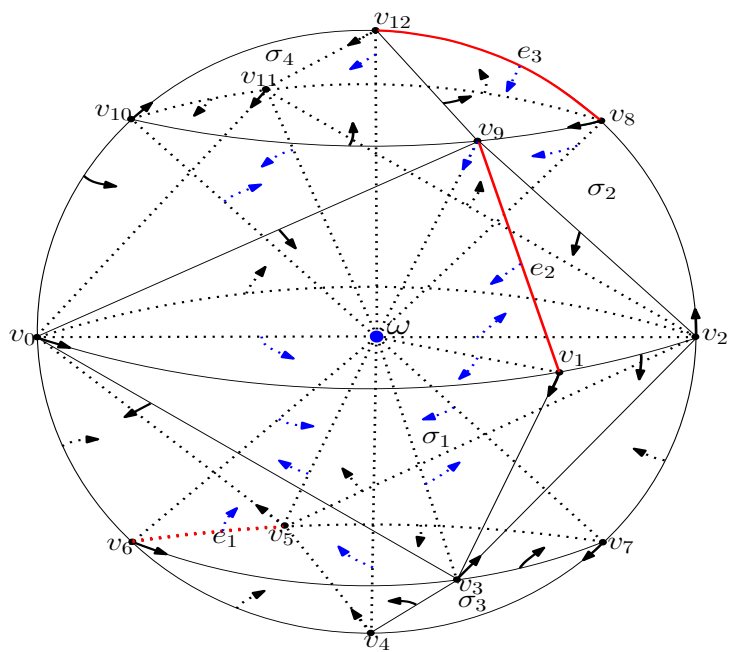}
			\caption{A perfect gradient vector field on $D^{3}$.}
			\label{fig:sphere0extend}
\end{figure}	
\end{proof} 

\begin{remark}
Note that we cannot extend the methods that we used in Theorem ~\ref{decompose3} and Theorem ~\ref{extend3} to
higher dimensional manifolds directly, since in the case of higher dimensional manifolds the separating sphere cannot be
classified only by the Euler characteristic.  Another problem that appears in higher dimensions is that the decomposition into a connected sum of prime factors is not necessarily unique, as in dimension $3$.  For example, 
The manifold $M=\mathbb{C} P^{2}\# \overline{\mathbb{C} P^{2}}\# \overline{\mathbb{C} P^{2}}$ is diffeomorphic   to $(S^{2}\times S^{2})\# \overline{\mathbb{C} P^{2}}$, while  $S^{2}\times S^{2}$ and $\mathbb{C} P^{2}\# \overline{\mathbb{C} P^{2}}$ 
are not even homotopy equivalent \cite{mcduff}.
\end{remark}

\bibliographystyle{amsplain}
\providecommand{\bysame}{\leavevmode\hbox
to3em{\hrulefill}\thinspace}
\providecommand{\MR}{\relax\ifhmode\unskip\space\fi MR }
 \MRhref  \MR
\providecommand{\MRhref}[2]{
  \href{http://www.ams.org/mathscinet-getitem?mr=#1}{#2}
 } \providecommand{\href}[2]{#2}

\end{document}